\documentclass[11pt,a4paper]{amsart}
\usepackage{amssymb,color}
\usepackage[english]{babel}          
\usepackage{verbatim}
\usepackage[T1]{fontenc}             

\usepackage{floatflt,graphicx}
\usepackage{color,xcolor}
\usepackage{enumerate}
\usepackage{animate}
\usepackage{amsmath}
\usepackage{amsthm}
\usepackage{epstopdf}

\newcounter{minutes}
\divide\time by 60
\newcounter{hours}
\multiply\time by 60 \addtocounter{minutes}{-\time}
\title
[Local convexity of metric balls ]
{Local convexity of metric balls}

\author{Parisa Hariri}
\address{Department of Mathematics and Statistics\\
  University of Turku\\ Turku, Finland}
\email{parisa.hariri@utu.fi}

\author{Riku Kl\'en }
\address{Department of Mathematics and Statistics\\
  University of Turku\\ Turku, Finland}
\email{riku.klen@utu.fi}

\author{Matti Vuorinen}
\address{Department of Mathematics and Statistics\\
  University of Turku\\ Turku, Finland}
\email{vuorinen@utu.fi}


\dedicatory{}

\theoremstyle{plain}
\newtheorem{thm}[equation]{Theorem}
\newtheorem{cor}[equation]{Corollary}
\newtheorem{lem}[equation]{Lemma}

\newtheorem{prop}[equation]{Proposition}
\newtheorem{conjecture}[equation]{Conjecture}

\theoremstyle{definition}
\newtheorem{defn}[equation]{Definition}

\theoremstyle{remark}
\newtheorem{rem}[equation]{Remark}

\newtheorem{nonsec}[equation]{}

\numberwithin{equation}{section}

\newcommand{\beq}{\begin{equation}}
\newcommand{\eeq}{\end{equation}}

\newcommand{\bequu}{\begin{eqnarray*}}
\newcommand{\eequu}{\end{eqnarray*}}

\newcommand{\bequ}{\begin{eqnarray}}
\newcommand{\eequ}{\end{eqnarray}}

\newcommand{\B}{\mathbb{B}^2}

\newcommand{\UH}{{\mathbb{H}^2}}
\newcommand{\BB}{\mathbb{B}^2}

\newcommand{\Hn}{ {\mathbb{H}^n} }

\newcommand{\Rn}{ {\mathbb{R}^n} }

\renewcommand{\Im}{{ \rm Im}\,}
\renewcommand{\Re}{{ \rm Re}\,}

\begin{document}

\def\thefootnote{}
\footnotetext{ \texttt{File:~\jobname .tex,
           printed: \number\year-\number\month-\number\day,
           \thehours.\ifnum\theminutes<10{0}\fi\theminutes}
} \makeatletter\def\thefootnote{\@arabic\c@footnote}\makeatother

\begin{abstract} We study local convexity properties of the triangular ratio metric balls in proper subdomains of the real coordinate space. We also study inclusion properties of the visual angle metric balls and related hyperbolic type metric balls in the complement of the origin and the upper half space.
\end{abstract}

\maketitle

\section{Introduction}\label{section1}


The hyperbolic metric has become an important tool in geometric function theory. It works well in simply connected subdomains of the complex plane, because we can use the Riemann mapping theorem to map such domains onto the unit disk, where explicit formulas are known \cite{b1,kl}. In higher dimensions ($n\ge 3$) no counterpart exists and thus there is need for other methods. One approach is to generalize the hyperbolic metric for higher dimensions in such a manner that the generalized metric is comparable with the hyperbolic metric when the domain is the upper half plane or the unit ball
$${\mathbb{H}^n} = \{ (x_1,\ldots,x_n)\in {\mathbb{R}^n}:  x_n>0 \}\,, \quad {\mathbb{B}^n}= \{ z\in {\mathbb{R}^n}: |z|<1 \}\,. $$
We call these generalizations hyperbolic type metrics. One of the first hyperbolic type metrics, the quasihyperbolic metric,
was introduced by Gehring and Palka in the 1970's \cite{gp}. Soon the quasihyperbolic metric
found numerous applications and nowadays it is a standard tool in geometric function theory, see e.g. \cite{gh}. During the past two decades many authors have introduced various other hyperbolic type metrics, \cite{himps}, \cite{klvw}, \cite{b}, \cite{s}, \cite{hmm}.

However, it is not clear which one of these hyperbolic type metrics is preferable for a specific application.
The natural line of research in this situation is to compare the geometries defined by two hyperbolic type metrics to each other. In this paper we study so called triangular ratio metric or $s$-metric, which is defined as follows for a domain $G \subset   \mathbb{R}^n$ and
$x,y \in G$:
\begin{equation}\label{sm}
s_G(x,y)=\sup_{z\in \partial G}\frac{|x-y|}{|x-z|+|z-y|}\in [0,1]\,.
\end{equation}
This metric has been studied in \cite{chkv,hklv}.
For a metric space $(G,m)$ we define the metric
ball for $x\inG$ and $r > 0$ by $B_m(x, r) = \{y\in G: m(x, y) < r\}\,.$
We study the metric balls defined by the triangular ratio metric.
The behaviour of a metric can be studied in many different ways. Our goal is to examine the geometric properties of the metric space $(G,s_G)$ by discussing local convexity properties of the metric balls $B_{s_G}(x,r)\,.$
We prove  that $B_{s_G}(x,r)\,, r \in(0,1],$
is always starlike with respect to $x\,$ and find the best constant $r_0$ such that $B_{s_G}(x,r)\,,$
$r\in(0,r_0)\,,$ is convex.
Similar local convexity results for other hyperbolic type metrics can be found in \cite{hklv,k2,k3,k4,krt,hps}. We also compare different hyperbolic type metrics by considering the inclusion of metric balls. For some other hyperbolic type metrics similar research is carried out in \cite{kv1,kv2}.

We next formulate our main results. 

\begin{thm}\label{stars}
Let $G\subsetneq\mathbb{R}^n$ be a domain. Then for all $z\in G$ and all $r\in (0,1]$, the $s$-ball $B_s(z,r)$ is Euclidean starlike with respect to $z\,.$
\end{thm}


In addition to the $s$-metric, we also study other metrics defined on subdomains of $\mathbb{R}^n\,,$ for example the quasihyperbolic metric $k_G$, the distance ratio metric $j_G$, the visual angle metric $v_G\,,$ and the point pair function $p_G$ which are defined in Section \ref{section5}. For these metrics we summarize our results in the following theorems.

\begin{thm}\label{thm:main3}
Let $x\in\mathbb{R}^n\setminus\{0\}\,,$ $r\in (0,r_m]\,,$ and $m\in \{j,k,|\cdot|,p,q,s\}\,.$ Then we can find the best possible radius $t=t(r)$ such that $B_m(x,t)\subset B_v(x,r)\,.$
\end{thm}
\begin{thm}\label{thm:main4}
Let $x\in \mathbb{H}^n\,$ and $r\in (0,\pi/2]\,.$ Then
\bequu
B^n(x+(\sec^2 r-1) x_n e_n, (\tan r) x_n)&\subset & B_v(x,r)\\
& \subset &  B^n \left( x+ (2x_n \tan^2 r)  e_n, 2x_n \frac{\tan r}{\cos r}  \right)\,,
\eequu
and the Euclidean balls are the best possible. Moreover
\[
 B^n(x,x_n \sin r) \subset B_v(x,r)\subset B^n \left( x, 2x_n\left( \frac{\tan r}{\cos r}+\tan^2 r\right)  \right)\,.
 \]
\end{thm}
This paper may be considered to be a continuation of the earlier studies \cite{hklv,k2,k3,k4,kms}. Our main results and their proofs suggest that similar results might be valid for other metrics as well and this offers ideas for further studies of the same topic, for instance for the Apollonian
or the Seittenranta metrics \cite{b,s}.


%
%

\section{Starlikeness and convexity of triangular ratio metric balls}\label{section3}


In this section we consider local convexity properties of $s$-metric balls. We start with $\mathbb{R}^n \setminus \{ 0 \}$ and generalize the results to proper subdomains of $\mathbb{R}^n$. Before studying local convexity properties we introduce preliminary results.

Given two points $x$ and $y$ in $\Rn$, the line segment between them is denoted by
$$
[x,y]=\{(1-t)x+ty\;:\; 0\le t\le1\}\,,
$$
and $\measuredangle(x, z, y)$ stands for the angle in
the range $[0, \pi]$ between the line segments $[x, z]$ and $[y, z]\,.$
\begin{defn}
Let $G\subsetneq \mathbb{R}^n$ be a domain and $x\in G\,.$ We say that $G$ is starlike with respect to $x$ if for every $y\in G\,,$ $[x,y]\subset G\,.$ The domain $G$ is strictly starlike with respect to $x$ if $G$ is bounded and each ray from $x$ meets $\partial{G}$ at exactly one point.
\end{defn}

{\bf Proof of Theorem \ref{stars}.} Without loss of generality we may assume $z =0\,.$
Fix $0<r\leq 1$. Let $y\in B_s(0,r)$. Since $s(y,0)<1$, $[0,y]\subset G$. Let $x\in[0,y]$ and $z\in\partial G$. Since
\[
|x|\leq |z|+|x-z| , \,\, |x|+|x-y|=|x|,\, \text{and}\,\, |y-z|\leq |y-x|+|x-z|,
\]
we have
\begin{eqnarray*}
\frac{|x|}{|z|+|x-z|}&\leq & \frac{|x|+|x-y|}{|z|+|x-z|+|x-y|}=\frac{|y|}{|z|+|x-z|+|x-y|}\\
&\leq & \frac{|y|}{|z|+|y-z|}\leq s(y,0)\,.
\end{eqnarray*}
Taking a supremum over all $z\in \partial G$ we thus obtain
\[
s(x,0)\leq s(y,0)< r\,,
\]
since $x$ is an arbitrary point in $[0,y]\subset B_s(0,r)$.\hfill $\square$

\bigskip

Next we continue the study of \cite{hklv} by considering the convexity of triangular ratio metric balls in a general subdomain of $\mathbb{R}^n$.
\begin{lem}\label{lem:convexity in punctured space}\cite[3.6, 3.8]{hklv}
  Let $x \in G = \mathbb{R}^n \setminus \{ 0 \}$ and $r \in (0,1)\,.$ Then $B_s(x,r)$ is (strictly) convex if and only if $r \le 1/2$ ($r<1/2$).
\end{lem}
\begin{thm}\label{thm:main2}
  Let $G \subsetneq \mathbb{R}^n$ be a domain, $x \in G$ and $r \in (0,1)\,.$ Then $B_s(x,r)$ is convex if $r \le 1/2\,.$
\end{thm}
\begin{proof}
  By \cite[(2.2)]{hklv} the ball $B_s(x,r)$ is the intersection of balls $B_{s_z}(x,r)\,,$ where $s_z$ is the triangular ratio metric in $\mathbb{R}^n \setminus \{ z \}\,,$ $z \in \partial G\,,$ and by Lemma \ref{lem:convexity in punctured space} each of these balls $B_{s_z}(x,r)$ is convex. The assertion follows as intersection of convex domains is a convex domain.
\end{proof}

\begin{figure}[h]
\begin{center}
     \includegraphics[height=.35 \textwidth]{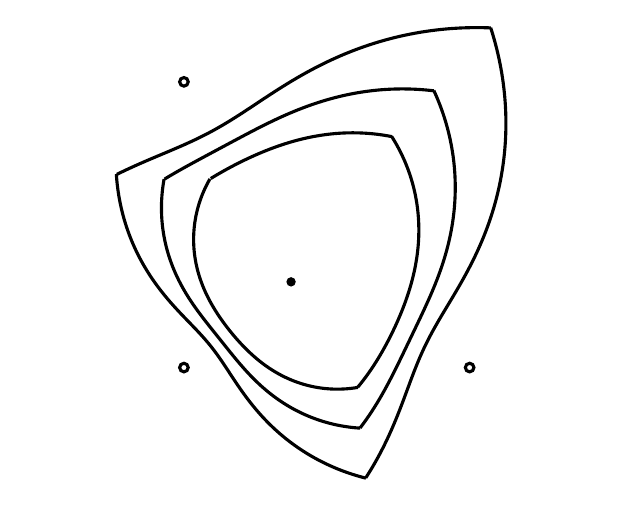}\hspace{5mm}
     \includegraphics[angle=90,height=.34 \textwidth]{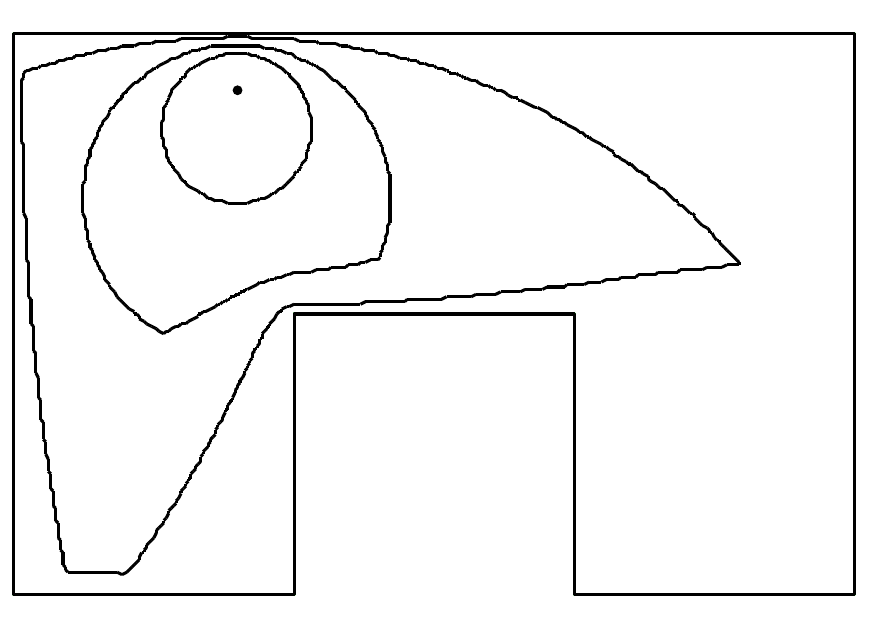}
\caption{Left: The $s$-metric disks $s(0.75e_1 + 0.6 e_2,r)$ in $\mathbb{R}^2 \setminus \{ 0,2e_1,2e_2 \}$ with $r=0.4,\, r=0.5,\, r=0.6\,.$ Right: $s$-metric disks in a polygonal domain.}
\end{center}
\end{figure}
Finally, we make the following conjecture for the balls of the point pair function, which is defined below in \ref{ppfun}.

\begin{conjecture}
  Let $x \in  \mathbb{R}^n \setminus \{ 0 \}$ and $r \in (0,1)\,.$ Then $B_p(x,r)$ is (strictly) convex if and only if $r \le \sqrt{2}-1\,.$
\end{conjecture}


\section{Formula for visual angle metric in $\Hn$}\label{section4}

For a domain $G\subsetneq \mathbb{R}^n\,,$ $n\geq 2\,,$ and $x,y\in G$ let
\begin{equation}\label{v}
v_G(x,y)=\sup\{\measuredangle (x,z,y): z\in\partial{G} \}\,.
\end{equation}
If $\partial G$ is not a proper subset of a line, then $v_G$ defines a metric on $G$, as shown in \cite[Lemma 2.8]{klvw}.

The supremum in \eqref{v} can be found by a geometric construction if $G=\mathbb{B}^2\,.$ Indeed by \cite[Theorem 1.2]{klvw}, by considering the two points $z_1$ and $z_2$ of intersection of the ellipses with foci at $0, x$ and $0,y\,,$ respectively, both with focal sum equal to $1\,,$ the formula for $v_{\mathbb{B}^2}(x,y)$ is just
\[
v_{\mathbb{B}^2}(x,y)=\max\{\measuredangle(x,z_1/|z_1|,y),\measuredangle(x,z_2/|z_2|,y)  \}\,.
\]
Here we find an analogue of this formula for $\mathbb{H}^2$ by finding the points of intersection of two parabolas with foci at $x$ and $y\,,$ respectively, and both with the real axis $\partial \mathbb{H}^2$ as the directrix, see Figure \ref{parabolas}.

\begin{figure}[h]\label{parafig}
  \begin{center}
     \includegraphics[width=14cm]{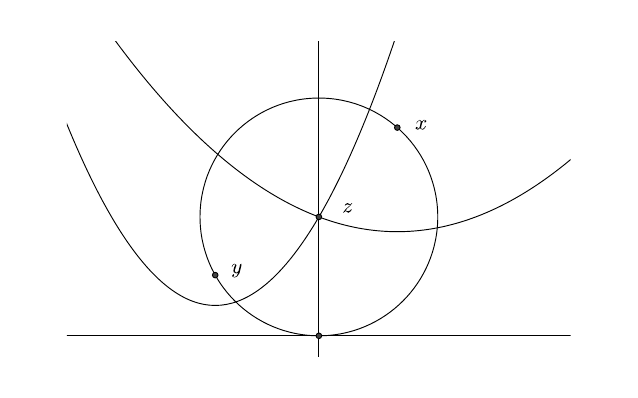}
  \caption{\label{parabolas} The point $z=(z_1,z_2)$ is the intersection of the parabola with focus $x$ and directrix $\partial \UH$ and the parabola with focus $y$ and directrix $\partial \UH\,.$ The extremal point $(z_1,0)$ in the definition of $v_{\UH}(x,y)$ can be found as the projection of $z$ to $\partial \UH\,.$}
  \end{center}
\end{figure}

For this purpose it is convenient to use horocycles. For two distinct points $x,y\in G$ where $G=\mathbb{B}^2$ or $G=\mathbb{H}^2$, a horocycle through $x,y$ is a Euclidean circle or line through $x$ and $y$ tangent to $\partial G\,.$

We consider the problem of finding the center points of two horocycles. These centers are the points of intersection of the parabolas with foci at $x$ and $y$ and directrix $\partial{\UH}\,.$ Therefore the formula for centers $z=(z_1,z_2)$ of these horocycles are given by
\[
|x-z|=z_2=|y-z|\,,
\]
we have
\[
  \begin{cases}
  (x_1-z_1)^2+(x_2-z_2)^2={z_2}^2\,, \\
   (y_1-z_1)^2+(y_2-z_2)^2={z_2}^2\,.
  \end{cases}
\]
By solving this system of quadratic equations we get
\[
z_1=\frac{x_2 y_1-x_1 y_2\pm\sqrt{x_2y_2}|x-y|}{x_2-y_2}\,, \quad x_2\neq y_2\,.
\]
If $x_2=y_2\,,$ then $z_1=(x_1+y_1)/2.$
In terms of this solution in the case $x_2\neq y_2\,,$ the possible extremal points for the visual angle metric are the two possible points of the form $(z_1,0)\in\partial{\UH}\,,$ and here we choose either $+$ or $-$ whichever corresponds to the smaller imaginary part.

Now by Figure \ref{vHfig}
\[
v_{\UH}(x,y)=\pi-\alpha-\beta\,,
\]
\[
\alpha=\arctan\left(\frac{x_2}{z_1-x_1}\right)\, \text{and}\,\,\, \beta=\arctan\left(\frac{y_2}{y_1-z_1}\right)\,.
\]
Therefore
\beq\label{vH}
v_{\UH}(x,y) \equiv \pi-\arctan\left(\frac{2\sqrt{x_2 y_2}|x-y| \pm (x_2+y_2)(x_1-y_1)}{(x_1-y_1)^2-4 x_2 y_2}\right) \quad (\textrm{mod\,} \pi)\,.
\eeq
Another formula for $v_{\UH}(x,y)$ can be derived from
\[
  z_2 = \frac{|x-y|}{2(x_2-y_2)^2} \left( |x-y|(x_2+y_2) \mp 2(x_1-y_1) \sqrt{x_2 y_2} \right)
\]
and the law of cosines together with the inscribed angle theorem.

\begin{figure}[h]
  \begin{center}
     \includegraphics[width=14cm]{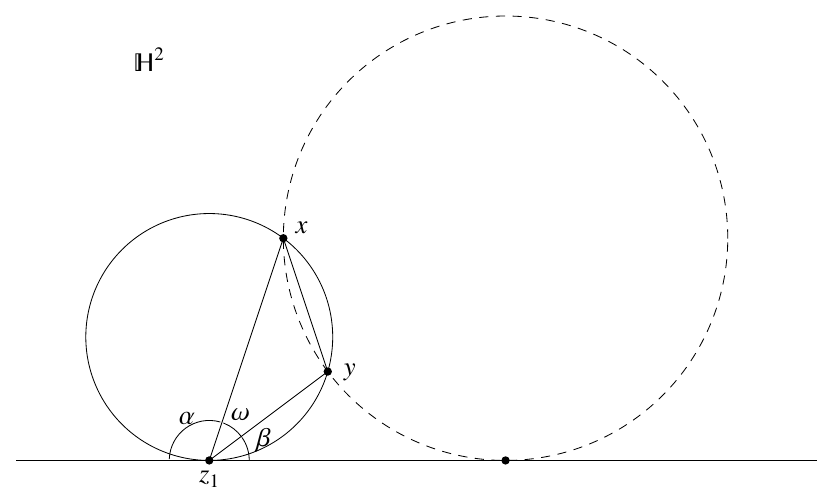}
  \caption{\label{vHfig} Two horocycles through $x$ and $y$ and the extremal point $(z_1,0)$ with $v_{\mathbb{H}^2}(x,y)=\measuredangle (x,z_1,y)\,.$}
  \end{center}
\end{figure}

\section{Inclusion properties of metric balls in $\Rn\setminus\{0\}$ and $\mathbb{H}^n$}\label{section5}

In this section we study inclusions of the visual angle metric balls and other metric balls. We begin by defining the metrics which we use.


\begin{nonsec}{\bf Quasihyperbolic metric.}
Let $G$ be a proper subdomain of ${\mathbb R}^n\,\,.$ For all $x,\,y\in G\,,$ the quasihyperbolic metric $k_G$ is defined as
$$k_G(x,y)=\inf_{\gamma}\int_{\gamma}\frac{1}{d(z,\partial G)}|dz|,$$
where the infimum is taken over all rectifiable arcs $\gamma$ joining $x$ to $y$ in $G$ \cite{gp}.
If we assume $x,y\in G=\Rn\setminus\{0\}$ and the angle $\varphi$ between the line segments $[0,x]$ and $[0,y]$ satisfies $0< \varphi< \pi$ then by \cite[3.11]{vu}
\beq\label{k}
k_G(x,y)=\sqrt{\varphi^2+\log^2{\frac{|x|}{|y|}}},\quad G=\Rn\setminus\{0\}\, .
\eeq

\end{nonsec}

\begin{nonsec}{\bf Distance ratio metric.}
For a proper open subset $G \subset {\mathbb R}^n\,$ and for all
$x,y\in G$, the  distance ratio
metric $j_G$ is defined as
$$
 j_G(x,y)=\log \left( 1+\frac{|x-y|}{\min \{d(x,\partial G),d(y, \partial G) \} } \right)\,.
$$
This metric was introduced by Gehring and Palka
\cite{gp} in a slightly different form and in the above form in \cite{vu0}. If confusion seems unlikely, then we also write $d(x)= d(x,\partial G)\,.$

\end{nonsec}


\begin{nonsec}{\bf Point pair function.}\label{ppfun}
We define for $x,y\in G\subsetneq\mathbb{R}^n$ the point pair function
$$p_G(x,y)=\frac{|x-y|}{\sqrt{|x-y|^2+4\,d(x)\,d(y)}}\,.$$
This point pair function was introduced in \cite{chkv} where it turned out to be a very
useful function in the study of the triangular ratio metric.
However, there are domains $G$ such that $p_G$ is not a metric: for instance this is the case if $G=\mathbb{B}^2\,,$ \cite[Remark 3.1]{chkv}.
\end{nonsec}

\begin{nonsec}{\bf Chordal metric.}
The chordal metric is defined by
\[
  \begin{cases}
 \displaystyle q(x,y) = \frac{|x-y|}{\sqrt{1+|x|^2}\sqrt{1+|y|^2}}&;\, x, y\in\Rn \,, \\
 \displaystyle  q(x,\infty)=\frac{1}{\sqrt{1+|x|^2}}\,.
  \end{cases}
\]
\end{nonsec}

\begin{prop}\label{vpun}
If $G= \Rn\setminus\{0\}\,,$ then the $v$-balls $B_v(x,r)\,,$ $x\in G\,,$ $r\in (0,\pi)\,,$ are angular domains with vertex at $0\,.$
\end{prop}

\begin{lem}\label{sv}
For all $x\in \Rn\setminus\{0\}\,,$ $r\in (0,\pi]$
\[
B_s \left( x,\sin{\frac{r}{2}} \right) \subset B_v(x,r)\,,
\]
and the radius $\sin{\frac{r}{2}}$ is best possible.
\end{lem}
\begin{proof}
By symmetry we may assume that $n=2$ and $x=e_1\,.$ Let $y=t e^{ir},\, t>0\,.$ Now
\[
s_G(x,y)=\frac{|x-y|}{|x|+|y|}=\frac{\sqrt{1+t^2-2t\cos{r}}}{1+t}=:f(t)\,.
\]
By Proposition \ref{vpun}, we want to minimize $f(t)\,.$ Now
\[
f'(t)=\frac{(t-1)(1+\cos{r})}{(1+t)^2\sqrt{1+t^2-2t\cos{r}}}\,,
\]
and $f'(t)=0$ if and only if $t=1\,,$ so $|y|=|x|$ and $s_G(x,y)=\frac{|x-y|}{2|x|}=\sin{\frac{r}{2}}\,.$
The sharpness follows from this argument.
\end{proof}

\begin{lem}\label{jv}
For all $ x\in\Rn\setminus\{0\}$ and $r\in (0,\pi]$
\[
B_j \left( x,\log \left( 1+2\sin{\frac{r}{2}} \right) \right) \subset B_v(x,r)\,,
\]
and the radius $\log(1+2\sin{\frac{r}{2}})$ is best possible.
\end{lem}

\begin{proof}
In the same way as in the proof of Lemma \ref{sv}, we may assume again that $n=2\,,$ $x=e_1\,$ and $y=t e^{ir},\, t>0\,.$ Now
\[
j_G(x,y)=\log \left( 1+\frac{|x-y|}{\min \{|x|, |y| \} } \right)=\log\left(1+\frac{\sqrt{1+t^2-2t\cos{r}}}{\min \{1,t\}}\right)\,.
\]
Define
\[
 f(t) =
  \begin{cases}
   \displaystyle \log\Big(1+\frac{\sqrt{1+t^2-2t\cos{r}}}{t}\Big)&\text{for}\,\,\, t\leq 1\,, \\ \\
    \displaystyle \log\Big(1+\sqrt{1+t^2-2t\cos{r}}\Big) & \,\text{for}\,\,\, t> 1\,.
  \end{cases}
\]
Computation yields
\[
 f'(t) =
  \begin{cases}
   \displaystyle \cfrac{t\cos{r}-1}{t(1-2t\cos r+t(t+ \sqrt{1+t^2-2t\cos{r}})} & \text{for}\,\,\, t\leq 1 \,,\\
   \displaystyle \cfrac{t-\cos{r}}{1+t^2-2t\cos{r}+ \sqrt{1+t^2-2t\cos{r}}}  & \text{for}\,\,\, t> 1\,.
  \end{cases}
\]
\medskip
By Proposition \ref{vpun}, the extremal case takes place when $t=1\,,$ and hence $|y|=|x|$ and $j_G(x,y)=\log \left( 1+\frac{|x-y|}{|x|} \right)=\log \left( 1+2\sin{\frac{r}{2}} \right)\,.$ Therefore $R=\log \left( 1+2\sin{\frac{r}{2}} \right)$ and the proof is complete. The sharpness follows from this proof.
\end{proof}


\begin{lem}\label{kv}
For all $ x\in\Rn\setminus\{0\}$ and $r\in (0,\pi]$
\[
B_k(x,r)\subset B_v(x,r),
\]
and the radius $r$ is best possible.
\end{lem}

\begin{proof}
We assume by symmetry that $n=2\,,$ $x=t e_1\,$ and  $y=e_1,\, t>0\,.$ Now
\[
k_G(x,y)=\sqrt{r^2+\log^2{\frac{|x|}{|y|}}}=\sqrt{r^2+\log^2 t}=:f(t)\,,
\]
and
\[
 f'(t) =\frac{\log t}{t\sqrt{r^2+\log^2 t}}\,.
\]
By Proposition \ref{vpun}, the extremal case happens when $t=1\,,$ so $|y|=|x|$ and $k_G(x,y)=r$ and the proof is complete. The sharpness follows from the proof.
\end{proof}

\begin{lem}\label{nv}
For all $ x\in\Rn\setminus\{0\}$ and $r\in (0,\pi]$, we have
\[
\mathbb{B}^n(x,R)\subset B_v(x,r)\,,\quad R =
  \begin{cases}
  \displaystyle |x|\sin{r} & \text{for}\,\,\, r\in (0,{\pi}/{2}]\,, \\
   \displaystyle |x|     & \text{for}\,\,\, r\in ({\pi}/{2},\pi]\,,
  \end{cases}
\]

and the radius $R$ is best possible.
\end{lem}

\begin{proof}
Fix $x\in\Rn\setminus\{0\}\,,$ and $y\in B_v(x,r)\,.$ We consider two cases. If $r\in ({\pi}/{2},\pi]$ then by Proposition \ref{vpun}, $|x-y|\leq |x|\,.$ If $r\in (0,{\pi}/{2}]$ then by the law of sines and Proposition \ref{vpun}, $|x-y|\leq |x|\sin{r}\,.$ The sharpness follows from the proof.
\end{proof}

\begin{lem}\label{qv}
For all $ x\in \Rn\setminus\{0\}$ and $r\in (0,\pi]$, we have
\[
B_q(x,R)\subset B_v(x,r)\,,\quad  R =\min\left\{ \frac{2|x|\sin{r/2}}{1+|x|^2},\frac{|x|}{\sqrt{1+|x|^2}} \right\}\,,
\]
and the radius $R$ is best possible.
\end{lem}

\begin{proof}
Fix $x\in\Rn\setminus\{0\}\,,$ and $y\in B_v(x,r)\,.$ We consider two cases. If $r\in ({\pi}/{2},\pi]$ then $$q(x,y)\leq q(x,0)=\frac{|x|}{\sqrt{1+|x|^2}}\,.$$ If $r\in (0,{\pi}/{2}]$ we may assume that $x=e_1\,,$ and let $y=t e^{ir},\, t>0\,.$ Then by the law of sines
 \[
 q(x,y)= \frac{|x-y|}{\sqrt{1+|x|^2}\sqrt{1+|y|^2}}=\frac{\sqrt{1+t^2-2t\cos{r}}}{\sqrt{2}\sqrt{1+t^2}}=:f(t)\,,
 \]
and
\[
f'(t)=\frac{(t^2-1)\cos{r}}{\sqrt{2}(1+t^2)^{3/2}\sqrt{1+t^2-2t\cos{r}}}\,.
\]
The extremal case takes place when $t=1\,,$ therefore $|x|=|y|\,,$ and
\[
q(x,y)=\frac{|x-y|}{1+|x|^2}=\frac{2|x|\sin(r/2)}{1+|x|^2}\,.
\]
The proof is complete because by Proposition \ref{vpun}, $B_v(x,y)$ is an angular domain. The sharpness follows from the proof.
\end{proof}


\begin{lem}\label{pv}
For all $ x\in \Rn\setminus\{0\}$ and $r\in (0,\pi]$, we have
\[
B_p(x,R)\subset B_v(x,r)\,,\quad  R =\frac{\sin(r/2)}{\sqrt{\sin^2(r/2)+1}}\,,
\]
and the radius $R$ is best possible.
\end{lem}

\begin{proof}
By symmetry we may assume that $n=2\,$ and $x=e_1\,.$ Let $y=t e^{ir},\, t>0\,.$ By the definition
\[
p_G(x,y)=\frac{|x-y|}{\sqrt{|x-y|^2+4|x| |y|}}=\sqrt{\frac{1+t^2-2t\cos{r}}{1+t^2-2t\cos{r}+4t}}=:f(t)\,.
\]
By Proposition \ref{vpun}, we want to minimize $f(t)\,.$ Now
\[
f'(t)=\frac{2 (t^2-1)}{\sqrt{1 + t^2 - 2 t \cos{r}} (1 + t (4 + t) - 2 t \cos{r})^{3/2}}\,,
\]
and $f'(t)=0$ if and only if $t=1\,,$ so $|y|=|x|$ and $$p_G(x,y)=\frac{|x-y|}{\sqrt{|x-y|^2+4|x|^2}}=\frac{\sin(r/2)}{\sqrt{\sin^2(r/2)+1}}.$$ The sharpness follows from the proof.
\end{proof}

{\bf Proof of Theorem \ref{thm:main3}.} The proof follows from Lemmas \ref{sv} -- \ref{pv}.
\hfill $\square$
 \begin{thm}\label{vbeb}
For all $x \in \mathbb{H}^n$ and $r \in (0,\pi/2)$, we have

  \bequ\label{vb}
    B_v(x,r) &\subset & B^n(x+ (2T^2)x_n e_n , 2T (\sqrt{T^2+1})x_n)\\
    &\subset & B^n(x, (2T \sqrt{T^2+1}+2T^2)x_n )\,,\nonumber
  \eequ
  where $T = | \tan (\pi-r) |\,.$ Moreover, the smaller Euclidean ball is the smallest possible containing $B_v(x,r)\,.$

  Note that \eqref{vb} is equivalent to
  \bequu
    B_v(x,r) &\subset & B^n \left( x+ (2x_n \tan^2 r)  e_n, 2x_n \frac{\tan r}{\cos r}  \right)\\
 &\subset &  B^n \left( x, 2x_n\left( \frac{\tan r}{\cos r}+\tan^2 r\right)  \right)\,.
  \eequu

\end{thm}

\begin{figure}[h]
\begin{center}
     \includegraphics[width=10cm]{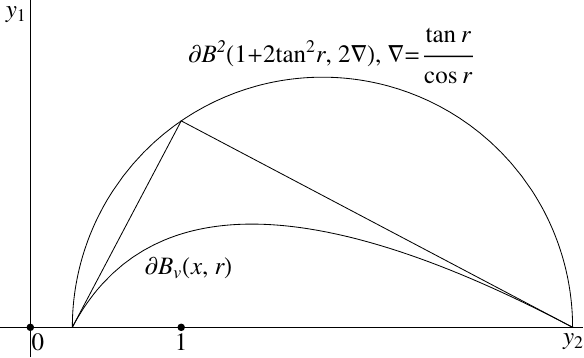}
\caption{ Proof of Theorem \ref{vbeb}.}
\end{center}
\end{figure}

\begin{proof}
It suffices to consider the case $n=2\,.$
For the first inclusion, let us fix $x=i\,.$ We claim that
\[
  B_v(i,r) \subset B^2 \left((1+2 \tan^2 r)i, 2 \frac{\tan r}{\cos r}  \right)\,.
\]
By \eqref{vH}, we have
\[
v_{\UH}(x,y) \equiv \pi-\arctan\left(\frac{2 \sqrt{y_2}\cdot \sqrt{y_1^2+(y_2-1)^2} \mp y_1(1+y_2)}{y_1^2-4 y_2}\right) \quad (\textrm{mod\,} \pi)\,.
\]
Writing $v_{\UH}(i,y)=r\,,$ we conclude that
\beq\label{f12}
y_1=\begin{cases}
\displaystyle\cot{r}(1+y_2-2\sqrt{y_2}\sec{r})=:f_1(y_2,r)\,, \\ 
\displaystyle\cot{r}(-1-y_2+2\sqrt{y_2}\sec{r})=:f_2(y_2,r)\,. \end{cases}
\eeq
This gives the equation of $\partial B_v(i,r)\,,$ for $y_2\in [b_1,b_2]\,.$  

Letting $f_1(y_2,r)=f_2(y_2,r)=0\,,$ we see that
\beq\label{b12}
y_2=\begin{cases}
\displaystyle 2(1-\sin{r})\sec^2{r}-1=:b_1\,, \\
\displaystyle 2(1+\sin{r})\sec^2{r}-1=:b_2\,. \end{cases} 
\eeq
Our next goal is to find the equation of $\partial B^2 \left((1+2 \tan^2 r)i, 2 \displaystyle \frac{\tan r}{\cos r}  \right)\,.$ Taking $ |y-(1+2 \tan^2 r)|= 2 \displaystyle \frac{\tan r}{\cos r}\,,$ gives
\[
y_1=\begin{cases}
\displaystyle -\sqrt{4y_2\sec^2{r}-(1+y_2)^2}=:g_1(y_2,r)\,, \\
\displaystyle \sqrt{4y_2\sec^2{r}-(1+y_2)^2}=:g_2(y_2,r)\,. \end{cases} 
\]
By symmetry, it is sufficient to show that $g_2(y_2,r)\geq f_2(y_2,r)\,.$ In order to prove this inequality, it is convenient to estimate the circular arc $g_2(y_2,r)\,,$ by a triangle with vertices $b_1, b_2$ and $(1,g_2(1,r))\,.$ We can do this estimation because
\[\frac{\partial^2}{\partial{y_2}^2}(g_2(y_2,r))=\frac{-4\sec^2{r}\tan^2{r}}{\left(4y_2\sec^2{r}-(1+y_2)^2\right)^{3/2}}\leq 0\,.
\]
Denote the above mentioned triangle by $T(y_2,r)$
\beq\label{l12}
T(y_2,r)=\begin{cases}
\displaystyle y_1=\frac{g_2(1,r)}{1-b_1}(y_2-b_1)=:l_1(y_2,r)\,,\quad b_1\leq y_2< 1\,, \\
\displaystyle y_1=\frac{g_2(1,r)}{1-b_2}(y_2-b_2)=:l_2(y_2,r)\,,\quad 1\leq y_2\leq b_2\,. \end{cases}
\eeq
We only need to show that $h(y_2,r)=T(y_2,r)-f_2(y_2,r)\geq 0\,.$

If $b_1\leq y_2 <1\,,$ then $$h(y_2,r)=\frac{1}{\cos{r}}(y_2-1+(y_2+1)\sin{r})-\cot{r}(2\sqrt{y_2}\sec{r}-1-y_2)\,.$$ But $$\displaystyle \frac{\partial^2(h(y_2,r))}{\partial y_2^2}=\frac{\csc{r}}{2y_2^{3/2}}>0\,,\,\, \text{for}\,\,\, \displaystyle 0\leq r\leq \pi/2\,.$$ Hence $h'(y_2,r)$ is an increasing function of $y_2\,.$
Moreover we claim that indeed $$h'(b_1,r)=\cot r + \sec r + \tan r-\frac{\csc r}{\sqrt{-1+2/(1+\sin r)}}=0\,.$$ To see this, it is sufficient to make the following observation
\[
\sqrt{\frac{1-\sin r}{1+\sin r}}=\frac{\csc r}{\cot r + \sec r + \tan r}\,.
\]
Now it is easy to check that for $0\leq r\leq \pi/2\,,$ both sides are equivalent to $\csc \left(\frac{\pi}{4}+\frac{r}{2}\right) \sin\left(\frac{\pi}{4}-\frac{r}{2}\right)\,.$ Thus $h'(y_2,r)\geq 0\,.$

Similarly we can show that
\[
h(b_1,r)=2 \sec r \left(\csc r-1-\cot
   r\sqrt{\frac{1-\sin r}{1+\sin r}} \right)=0\,.
\]
To see this, it suffices to show that
\[
\frac{1-\sin r} {\cos r}=\sqrt{\frac{1-\sin r}{1+\sin r}}\,.
\]
It follows easily that for $0\leq r\leq \pi/2\,,$ both sides are equivalent to
$$\csc \left(\frac{\pi}{4}+\frac{r}{2}\right) \sin\left(\frac{\pi}{4}-\frac{r}{2}\right)\,.$$
Hence $h(y_2,r)\geq 0\,.$
In the same manner for $1\leq y_2\leq b_2\,,$
\[
h(y_2,r)=(1+y_2+\csc r - y_2 \csc r)\tan r- \cot r(-1-y_2+2\sec r \sqrt{y_2})\,.
\]
But $$\frac{\partial^2(h(y_2,r))}{\partial y_2^2}=\frac{\csc{r}}{2y_2^{3/2}}>0\,,\,\,\text{for}\,\,\, 0\leq r\leq \pi/2\,,$$ and
\[
h'(b_2,r)=\tan r+\cot r-\sec r-\frac{\csc r}{\sqrt{2 \sec
   ^2r+2 \tan r \sec r-1}}=0\,.
\]
To see this, it is enough to show that
\[
\sqrt{2 \sec ^2 r+2 \tan r \sec r-1}=\frac{\csc r}{\tan r+\cot r-\sec r}\,,
\]
and it is easy to check that for $0\leq r\leq \pi/2\,,$ both sides are equivalent to
$\csc \left(\frac{\pi}{4}-\frac{r}{2}\right) \sin\left(\frac{\pi}{4}+\frac{r}{2}\right)\,.$
Therefore $h'(y_2,r)\geq 0\,.$ We next show that
\[
h(b_2,r)=2 \sec r \left(\csc r-\cot r \sqrt{2 \sec r
   (\tan r+\sec r)-1}+1\right)=0\,.
\]
To see this we need to show that
\[
\sqrt{2 \sec r
   (\tan r+\sec r)-1}=(1+\csc r)\tan r\,.
\]
In the same manner as in the previous part, we can show that for $0\leq r\leq \pi/2\,,$ both sides are equivalent to
$\csc \left(\frac{\pi}{4}-\frac{r}{2}\right) \sin\left(\frac{\pi}{4}+\frac{r}{2}\right)\,.$
Hence $h(y_2,r)\geq 0\,.$

An easy computation shows that $g_2(b_1,r)
=g_2(b_2,r)=0\,.$ Hence the Euclidean ball $ B^2 \left((1+2 \tan^2 r)i, 2\displaystyle \frac{\tan r}{\cos r}  \right)$ is the smallest possible ball containing $B_v(x,r)\,,$ and this completes the proof for the first inclusion.

For the second inclusion, let
\[
y\in \partial  B^2 \left( x+ (2x_2 \tan^2 r)  e_2, 2 \frac{\tan r}{\cos r} x_2 \right)\,.
\]
It follows that
\[
|y- x_1e_1 - (1+2 \tan^2 r) x_2 e_2 |\leq 2 \frac{\tan r}{\cos r} x_2 \,.
\]
Therefore
\bequu
|y- x| &\leq & |y- x_1e_1 - (1+2 \tan^2 r) x_2 e_2 |+|-x+x_1e_1 + (1+2 \tan^2 r) x_2 e_2 |\\
&\leq & 2x_2 (\frac{\tan r}{\cos r}  + \tan^2 r)\,,
\eequu
and $y\in \partial B^2\left(x,2x_2 ( \displaystyle\frac{\tan r}{\cos r}  + \tan^2 r)\right)\,.$
\end{proof}

\begin{figure}[h]

\begin{center}

     \includegraphics[width=10cm]{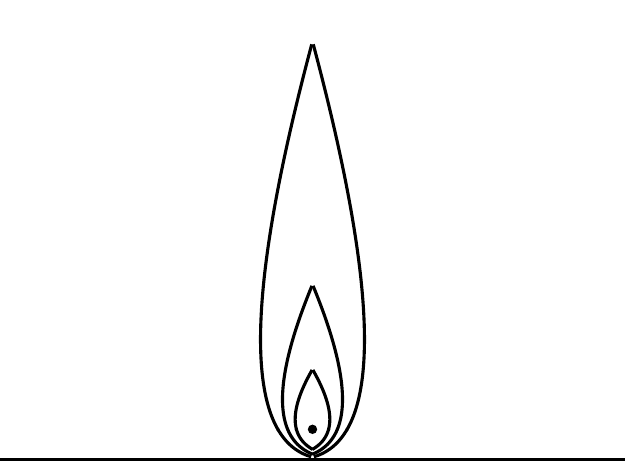}

\caption{ $\partial B_v(i,r)\,,$ in the upper half space, for $r= \pi/6, \pi/4, \pi/3\,.$  }

\end{center}

\end{figure}
\begin{rem}\label{rmk5.18}
By the proof of Theorem \ref{vbeb} we observe that $\partial B_v(x,r)$ is not smooth for $0\leq r\leq \pi\,.$ By \eqref{f12},
\[
f_2'(b_1^{+},r)=\cot r \left(\frac{\sec r}{\sqrt{2 \sec ^2r-2 \tan r
   \sec r-1}}-1\right)\,,
\]
and
\[
f_1'(b_1^{+},r)=-\cot r \left(\frac{\sec r}{\sqrt{2 \sec ^2r-2 \tan r
   \sec r-1}}-1\right)\,.
\]
Hence we see that at the point $y_2=b_1\,,$ the derivative does not exist and $\partial B_v(x,r)$ is not smooth.
\end{rem}

\begin{rem}
The proof of Theorem \ref{vbeb} gives more, namely, the lines $l_1(y_2,r)$ and $l_2(y_2,r)$ are tangent to the $\partial B_v(x,r)\,.$ To see this we first compute the slope of the tangent lines to $B_v(x,r)$ at the points $b_1, b_2\,.$ We have by Remark \ref{rmk5.18}
\[
f_2'(b_1^{+},r)=\cot r\left(\frac{\sec r}{\sqrt{2 \sec ^2r-2
   \tan r \sec r-1}}-1\right).
\]
Next, by \eqref{l12}, the slope of the line $l_1$ is $$\displaystyle m_1=\frac{\tan r}{-\sec ^2r+\tan r \sec r+1}\,.$$ We claim that $m_1=f_2'(b_1^{+},r)\,.$ To see this it is enough to show that
\[
\frac{\tan^2 r+1-\sec^2 r+\sec r\tan r}{\sec r}=\frac{-\sec ^2r+\tan r \sec r+1}{\sqrt{2 \sec ^2r-2
   \tan r \sec r-1}}\,.
\]
It is easy to check that for $0\leq r\leq \pi/2\,,$ both sides are equivalent to $$\displaystyle\cos\left(\frac{r}{2}\right) \csc\left(\frac{\pi}{4}-\frac{r}{2}\right) \csc\left(\frac{\pi}{4}+\frac{r}{2}\right) \sin\left(\frac{r}{2}\right)\,.$$

Similarly, substituting $b_2$ from \eqref{b12}, gives
\[
f_2'(b_2^{+},r)=\cot r \left(\frac{\sec r}{\sqrt{2 \sec ^2r+2 \tan r
   \sec r-1}}-1\right)\,.
\]
Next, by \eqref{l12}, the slope of line $l_2$ is $$m_2=-\frac{\tan r}{\sec ^2r+\tan r \sec r-1}\,.$$ We claim that $m_2=f_2'(b_2^{+},r)\,.$ To see this, it is enough to show that
\[
\frac{-\tan^2 r-1+\sec^2 r+\sec r\tan r}{\sec r}=\frac{\sec ^2r+\tan r \sec r-1}{\sqrt{2 \sec ^2r+2 \tan r
   \sec r-1}}\,.
\]
It is easy to check that for $0\leq r\leq \pi/2\,,$ both sides are equivalent to $$\cos\left(\frac{r}{2}\right) \csc\left(\frac{\pi}{4}-\frac{r}{2}\right) \csc\left(\frac{\pi}{4}+\frac{r}{2}\right) \sin\left(\frac{r}{2}\right)\,.$$
\end{rem}

\begin{thm}\label{vbeb2}
For all $x \in \mathbb{H}^n$ and $r \in (0,\pi/2)$, we have

  \[
    B^n(x+(\sec^2 r-1) x_n e_n, (\tan r) x_n)\subset B_v(x,r)\,.
  \]
  \end{thm}
\begin{proof}
By symmetry, it suffices to consider the case $n=2\,.$

Let us fix $x=i\,.$ We claim that
\[
  B^2(i\sec^2 r , \tan r)\subset B_v(x,r)\,.
\]
By \eqref{f12}, the equation of $\partial B_v(i,r)\,$ is as follows:
\[
y_1=\begin{cases}
\displaystyle \cot{r}(1+y_2-2\sqrt{y_2}\sec{r})=:f_1(y_2,r)\,, \\
\displaystyle \cot{r}(-1-y_2+2\sqrt{y_2}\sec{r})=:f_2(y_2,r)\,. \end{cases} 
\]

Next we find the equation of $\partial B^2(i\sec^2 r, \tan r)\,.$ Taking $|y-\sec^2 r|= \tan r$ gives
\[
y_1=\begin{cases}
\displaystyle -\sqrt{-(y_2 - \sec^2 r)^2 + \tan^2 r}=:g_1(y_2,r)\,, \\
\displaystyle \sqrt{-(y_2 - \sec^2 r)^2 + \tan^2 r}=:g_2(y_2,r)\,. \end{cases} 
\]
By symmetry, it is sufficient to show that $f_2(y_2,r)\geq g_2(y_2,r)\,.$ To see this we only need to show that $\displaystyle h(y_2,r)=:f_2(y_2,r)^2-g_2(y_2,r)^2 \geq 0\,.$ But $$\displaystyle \frac{\partial^2 h(y_2,r)}{\partial{y_2^2}}=\frac{1}{y_2^{3/2}}(2 y_2^{3/2} + \cos r - 3 y_2 \cos r) \csc^2 r\,.$$ Denote by $$\displaystyle h_2(y_2,r)=2 y_2^{3/2} + \cos r - 3 y_2 \cos r\,.$$ But 
\[
\frac{\partial^2 h_2(y_2,r)}{\partial{y_2^2}}=\frac{3}{2\sqrt{y_2}}> 0\,,\,\,\text{and}\quad \frac{\partial h_2(\cos^2 r,r)}{\partial{y_2}} =0\,.  
\]
 Moreover $h_2(\cos^2 r, r)>0$ for $r\in (0,\pi/2)\,.$ Therefore $\frac{\partial^2 h(y_2,r)}{\partial{y_2^2}}\geq 0.$ It is easy to check that $\frac{\partial h(\sec^2 r,r)}{\partial{y_2}}=0\,,$ and hence $h(y_2,r)\geq 0\,.$
\end{proof}

\begin{thm}\label{vbeb3}
For all $x \in \Hn\,$ and $r \in (0,\pi/2)$, we have
  \beq\label{bv}
    B^n(x,x_n \sin r) \subset B_v(x,r)\,.
  \eeq
\end{thm}
\begin{proof}
Let $\lambda = \sin r$ and $y\in  B^n(x,\lambda x_n)\,.$ By domain monotonicity of the $v$-metric, we have for $B_x=B^n(x, x_n)$ the inequalities
  \[
    v_{\Hn}(x,y)\leq v_{B_x}(x,y)\leq \arcsin{\lambda}\,,
  \]
  where the last inequality follows by \cite[3.3]{klvw}.
\end{proof}


{\bf Proof of Theorem \ref{thm:main4}.} The proof follows from Theorems \ref{vbeb}, \ref{vbeb2} and \ref{vbeb3}.
\hfill
$\square$


\bigskip
{\bf Acknowledgements.} The research of the first author was supported by UTUGS, The Graduate School of the University of Turku.



\end{document}